\documentclass[letterpaper, 10 pt, conference]{ieeeconf}  

\usepackage{graphicx}      
\usepackage{soul}
\usepackage{latexsym}
\usepackage{amsmath}
\usepackage{amsfonts}
\usepackage{epstopdf}
\usepackage{amssymb}
\usepackage[english]{babel}
\usepackage{bm}
\usepackage{graphicx,subfig}
\newtheorem{example}{Example}
\newtheorem{definition}{Definition}
\newtheorem{algorithm}{Algorithm}

\newtheorem{theorem}{Theorem}
\newtheorem{lemma}[theorem]{Lemma}
\newtheorem{proposition}[theorem]{Proposition}
\newtheorem{corollary}[theorem]{Corollary}
\newtheorem{remark}[theorem]{Remark }

\usepackage{color}
\usepackage{algorithm}
\usepackage[noend]{algpseudocode}
 
\usepackage{xcolor}
\definecolor{verde}{rgb}{0,0.7,0}

\usepackage[colorlinks=true,linkcolor=blue,urlcolor=blue,citecolor=blue]{hyperref}

\newcommand{\be}{\begin{equation}}
\newcommand{\ee}{\end{equation}}

\newcommand{\bea}{\begin{eqnarray}}
\newcommand{\eea}{\end{eqnarray}}
\newcommand{\bean}{\begin{eqnarray*}}
\newcommand{\eean}{\end{eqnarray*}}

\IEEEoverridecommandlockouts                              

\begin{document}

\title{\LARGE \bf Algebraic and Graph-Theoretic Conditions for the Herdability  of Linear Time-Invariant Systems}

 \author{Giulia De Pasquale and Maria Elena Valcher
 \thanks{G. De Pasquale and M.E. Valcher are with
 the Dipartimento di Ingegneria dell'Informazione
 Universit\`a di Padova, 
    via Gradenigo 6B, 35131 Padova, Italy, e-mail:  \texttt{giulia.depasquale@phd.unipd.it, meme@dei.unipd.it}.
This Paper is an extended version of the paper G.
De Pasquale, M.E. Valcher, "Algebraic and Graph-Theoretic Conditions for the Herdability  of Linear Time-Invariant Systems", Proceedings of the 60th IEEE Conference on Decision and Control (CDC) 2021, Austin, Mexico, USA.
}
   } 
 \maketitle

\begin{abstract}                          
In this paper we investigate a relaxed concept of controllability, known in the literature as herdability, namely the capability of a system to be driven towards 
the (interior of the) positive orthant.   Specifically, we investigate herdability for linear time-invariant systems, both from an algebraic perspective and based on  the graph representing the systems interactions. In addition, we focus on  linear state-space models corresponding to matrix pairs $(A,B)$ in which the matrix $B$ is a selection matrix that determines the  leaders in the network, and we show that  the weights that  followers  give to the leaders  do not affect the herdability of the system. We then focus on the herdability problem for   systems with a single leader in which  interactions are symmetric and the network topology is acyclic, in which case an algorithm for the leader selection is provided. In this context, under some additional conditions on the mutual distances,
 necessary and sufficient conditions for the herdability of the overall system are given.
  \end{abstract}
\section{Introduction}
There are many applications in control systems theory in which requiring that the system is controllable, namely 
that 
the system state can be driven towards any point in the state space, is unnecessary, due to the nature of the involved application.
This is what happens, in particular, when dealing with networked systems \cite{antsaklis2007special,LiZhuDing}. This kind of systems, in fact, comes into play in many applications related to biology \cite{Jacquez}, chemistry \cite{comp_gen}, sociology \cite{sociology1}, neuroscience \cite{gupta}, etc. In all these contexts, asking whether the state can be brought to an arbitrary point of the state space may lead to unnecessarily restrictive conditions  on the system model. For example, in the study of a chemical reaction  it makes no sense to 
impose that variables representing solvent concentrations may reach
 negative values. In  these cases it becomes of interest to study   a weaker concept with respect to the one of controllability, known in the literature as herdability \cite{Ruf_Shamma,Ruf_arxiv}.
Herdability indicates the capability  of a system to be driven towards the positive orthant.
In mathematical terms,  a continuous-time linear time-invariant  system $$\dot{\bf x}(t) = A {\bf x}(t) + B {\bf u}(t),$$ with $A \in {\mathbb R}^{n \times n}, B \in {\mathbb R}^{n \times m}$,   is \textit{herdable} if, for every initial condition ${\bf x}(0)$, there exists a control input that drives all the   state variables over a positive threshold.

In this work we investigate the herdability property of linear time-invariant (LTI) systems both from an algebraic and a topological perspective. In particular, we  consider systems whose associated network has cooperative and competitive relationships and we   try to understand when herdability can be deduced by simply looking at the type of relationships (namely at their signs) rather than at their values. 

Controllability of multi-agent systems is a well-established research field.
Extensive results have been obtained for sign controllable and structurally controllable networked systems \cite{struct_controllability,ParlangeliNotarstefano,graph_controllability,sign_controllability}.  In \cite{struct_controllability} the controllability of LTI systems defined on a signed graph is discussed, by focusing on the controllability analysis of positive and negative eigenvalues of systems whose matrices share the same sign pattern. In \cite{ParlangeliNotarstefano,sign_controllability} complete controllability conditions for matrix pairs $(A,B)$ are deduced based on some graph properties.
Finally, in \cite{graph_controllability} the controllability of multi-agent systems is studied by assuming that a subset of agents represent the leaders, while the remaining ones execute local protocols. It is also shown how the symmetry of the network structure affects the controllability property of the overall system.

Herdability of networked systems, on the other hand,  is a much more recent line of research, as witnessed by the recent works \cite{Meng_herd,Ruf_Shamma,Ruf_arxiv}, \cite{She_herd,She_automatica}.
In \cite{She_herd} the herdability property of dynamic leader-follower signed networks is studied from a topological point of view, under the assumption that the leaders are endowed with exogenous control inputs, thus developing sufficient conditions for herdability based on 1-walks and 2-walks on the graph. In \cite{Ruf_arxiv}  the topology and sign distribution of  the underlying graph of an LTI system in related to its herdability property, while in \cite{Ruf_Shamma}  herdability of subsets of nodes in a graph is investigated, with a special focus on the herdability of directed out-branching rooted graphs with a single input. In the paper \cite{She_automatica}, the controllable subspace of a system is characterized based on graph partitions, and sufficient conditions for the system herdability are deduced. The concept of quotient graph is exploited in order to deduce the herdability of the original graph.

Inspired by the methods exploited in the study of the controllability of networked systems and motivated by the practical need to relax the controllability property, especially in the context of networked systems with cooperative and competitive interactions, we introduce here some sufficient conditions for the herdability of  LTI systems. We   focus, in particular, on systems with leader-follower networks and undirected acyclic graphs. We also introduce some algorithms to check   some conditions ensuring the system herdability with a special focus on tree topologies. Finally, we propose a leader selection strategy aimed at guaranteeing the system herdability.

In detail, Section \ref{s2} provides some sufficient conditions for the herdability of a generic matrix pair $(A,B)$, based on the algebraic structure of the system. In Section \ref{s3} herdability of  leader-follower networks is studied, while Section \ref{s4} focuses on systems whose  associated graph has   a single leader and tree topology. Section \ref{s5} concludes the paper.
\medskip

{\bf Notation}.
Given   $k, n\in \mathbb{Z}$, with $k <n$,   the symbol   $[k,n]$   denotes the  integer set  $\{k, k+1, \dots, n\}$.
The $(i,j)$-th entry of a matrix $A$ is denoted 
 by $[A]_{ij}$, while the $i$-th entry of a vector ${\bf v}$ by $[{\bf v}]_i$.
 The notation $M= {\rm diag} \{M_1, M_2,  \dots,  M_n\}$ indicates a block  diagonal matrix  with diagonal blocks $M_1, M_2,  \dots,  M_n$.
We let   ${\bf e}_i$ denote the $i$-th vector of the canonical basis of $\mathbb{R}^n$,  where the dimension $n$ will be   clear from the context. 
  Accordingly, $M {\bf e}_j$ denotes the $j$-th column of $M$, 
  ${\bf e}_i^\top M$ the $i$-th row of $M$, and ${\bf e}_i^\top M {\bf e}_j$ the $(i,j)$-th entry of $M$.  Every nonzero multiple of a canonical vector is called {\em monomial vector}.
The vectors ${\bf 1}_n$ and ${\bf 0}_n$ denote the $n$-dimensional vectors whose entries are all $1$ or $0$, respectively.
 Similarly, the symbol ${\bf 0}_{p \times m}$ denotes the $p\times m$ matrix with all zero entries.



   Given a vector ${\bf v} \in {\mathbb R}^n$, the set    
   $\overline{\rm ZP}({\bf v}) = \{ i \in [1,n]: [{\bf v}]_i \neq 0\}$  denotes the  
   {\em non-zero pattern} of ${\bf v}$  \cite{TACconPaolo}.  Similarly, we can define the   non-zero pattern of a matrix $A$. A nonzero vector ${\bf v}$ is said to be {\em unisigned} \cite{Ruf_arxiv} if all its nonzero entries have the same sign.
If ${\bf v}$ is a unisigned vector, then by ${\rm sign}({\bf v})$ we mean the common sign of its nonzero entries.
 In other words, ${\rm sign}({\bf v})=1$ if the nonzero entries of ${\bf v}$ are positive, while 
${\rm sign}({\bf v})=-1$ if the nonzero entries of ${\bf v}$ are negative.

 Given a matrix $A\in \mathbb{R}^{n \times m}$, the notation ${\rm Im (A)}$ denotes the image of the matrix $A$.
 A matrix  (in particular, a vector)
 $A$ is   {\em nonnegative}  (denoted by  $A \ge 0$) \cite{BookFarina}  if all its entries are nonnegative.
  $A$ is   {\em strictly positive} (denoted by  $A \gg 0$) if all its entries   are positive.
  A matrix $P\in {\mathbb R}^{n \times n}$ is a {\em permutation matrix} if its columns are a permuted version of the columns of the identity matrix $I_n$.

%
%
%

 Given a set ${\mathcal S}$, the {\em cardinality} of ${\mathcal S}$ is denoted by $|{\mathcal S}|$.
To any matrix   $A\in {\mathbb R}^{n \times n}$, we associate the {\em signed and weighted directed graph}
${\mathcal G}(A)= ({\mathcal V}, {\mathcal E}, A),$
where ${\mathcal V} =[1,n]$ is the set of nodes.
The set ${\mathcal E}\subseteq {\mathcal V}\times {\mathcal V}$ is  the set of arcs  (edges)
connecting  the nodes, while the matrix $A\in {\mathbb R}^{n \times n}$ is the adjacency matrix of the graph. There  is an arc $(j, i) \in \mathcal{E}$ from $j$ to $i$, if and only if $[A]_{ij} \ne 0$.  When so, $[A]_{ij}$ is the {\em weight} of the arc.

  A sequence of $k$ consecutive arcs 
 $(j, j_2),(j_2, j_3),$ $\dots, (j_k,i)\in {\mathcal E}$ is a {\em walk}  
 of length $k$ 
 from $j$  to 
 $i$.
A walk from $j$   to $i$ is said to be {\em positive (negative)} if the product of the weights of the edges that compose the walk is positive (negative). 
  A directed graph ${\mathcal G}(A)$ is {\em strongly connected} if for every pair of vertices $i,j\in {\mathcal V}$ there exists a walk from $i$ to $j$.
A {\em minimum walk} from $j$   to $i$ is a walk of minimum length connecting the two nodes. We define the   {\em distance} $d(j,i)$ from the node $j$ to the node $i$
as the length of the minimum walk from $j$ to $i$. The distance    $d(j,{\mathcal I})$ from the node $j$ to the set of nodes $\mathcal{I}$ is the minimum among all the distances $d(j,i),$ $i \in {\mathcal I}$. Similarly, the distance    $d({\mathcal I},j)$ from the set of nodes ${\mathcal I}$ to the vertex $j$ is   the minimum among all the distances $d(i,j),$ $i \in {\mathcal I}$.
\\ Given a node $i \in \mathcal{V}$,
we define the {\em out-neighborhood} of node $i$ as the set of nodes $j$ such that $d(i,j)=1$, namely ${\rm Out}(i) = \{j \in \mathcal{V}: (i,j)\in \mathcal{E}\}$. We define the {\em positive out-neighborhood}  of node $i$ as   the set of nodes $j$ such that $(i,j)$ is an arc of ${\mathcal G}(A)$ of positive weight, namely ${\rm Out}_+(i) = \{j \in \mathcal{V}: [A]_{ji}>0\}$. The definition of {\em negative out-neighborhood} of a node is analogous.
The out-neighborhood can be also defined for a set of nodes $\mathcal{I}\subset {\mathcal V}$ as 
 ${\rm Out}({\mathcal I}) = \{j \in \mathcal{V}\setminus {\mathcal I}: (i,j)\in \mathcal{E}, \exists\  i \in {\mathcal I} \}$. The definitions of ${\rm Out}_{+}({\mathcal I})$ and ${\rm Out}_{-}({\mathcal I})$ are  analogous.\\
 If $A$ is a symmetric matrix, namely $A=A^\top$, the graph ${\mathcal G}(A)$ is (signed, weighted and) undirected, and all previous concepts (in particular, the concepts of walk and distance) become symmetric.
\medskip

A graph ${\mathcal G}(A)$ is said 
to be {\em structurally balanced} if all    its nodes can be partitioned into two disjoint subsets ${\mathcal V}_1$ and ${\mathcal V}_2$ in a way such that $\forall i,j \in    {\mathcal V}_p$, $p \in \{1,2\}$, $[A]_{ij} \geq 0$ and $\forall i \in \mathcal{V}_p$ and $\forall j \in {\mathcal V}_q$, $p,q \in \{1,2\},$ $p \neq q$, it holds that $[A]_{ij}\leq 0$. Note that if ${\mathcal V}_1=[1,m]$, while ${\mathcal V}_2 = [m+1,n]$,   the matrix $A$ can be block partitioned as
$$A = \begin{bmatrix} A_{11} & A_{12}\cr A_{21} & A_{22}\end{bmatrix},$$
where $A_{11}\in {\mathbb R}^{m\times m}$ and $A_{22}\in {\mathbb R}^{(n-m)\times (n-m)}$ are nonnegative matrices, while $A_{12}$ and $A_{21}$ are nonpositive matrices (i.e., the opposite of nonnegative matrices).
\medskip

\section{Sufficient conditions for herdability of general pairs $(A,B)$}\label{s2}

 The concept of herdability of linear and time-invariant state space models
 described by a matrix pair $(A,B)$, with $A \in {\mathbb R}^{n \times n}$ and $B \in {\mathbb R}^{n \times m}$,
 has been defined in various ways \cite{Ruf_Shamma,Ruf_arxiv,She_herd}. 
In this paper we are   interested 
 in the behavior of all state variables, rather than in the behavior of a subset of them.   Consequently, 
 we assume the following   definition
 (which is equivalent to Definition 3 in \cite{Ruf_arxiv}).
\medskip
 
\begin{definition}  Given a (continuous-time or discrete-time)   (linear and time-invariant) state space model of dimension $n$ with $m$ inputs,
described by a pair $(A,B), \ A\in {\mathbb R}^{n\times n}$ and $B\in {\mathbb R}^{n\times m}$, 
the system (the pair) is said to be {\em (completely) herdable} if for every ${\bf x}(0)$ and every $h> 0$, there exists a time $t_f > 0$ and an input $u(t), t\in [0,t_f),$ that drives the state of the system from ${\bf x}(0)$ to ${\bf x}(t_f) \ge h {\bf 1}_n$.
\end{definition}
\smallskip

Both in the continuous-time case and in the discrete-time case, herdability reduces to a condition on the controllability matrix associated with the pair $(A,B)$.
\medskip

\begin{proposition} [Corollary 1, \cite{Ruf_arxiv}] A pair $(A,B), A\in {\mathbb R}^{n\times n}$ and $B\in {\mathbb R}^{n\times m}$, is herdable if and only if 
there exists a strictly positive vector belonging to ${\rm Im}({\mathcal R}(A,B))$, where 
\be
{\mathcal R}(A,B) := \begin{bmatrix} B & AB & A^2B & \dots & A^{n-1} B \end{bmatrix}
\label{reach_mat}
\ee
is the {\em controllability matrix} of the pair $(A,B)$.
\end{proposition}
\smallskip

 Clearly, every reachable pair $(A,B)$ is herdable, but the converse is not true. Also, if ${\mathcal R}(A,B)$ has zero rows then the problem is clearly not solvable.
So, in the following we will investigate herdability  by assuming that ${\mathcal R}(A,B)$ is devoid of zero rows and
${\rm Im} ({\mathcal R}(A,B))$ is a proper subset of ${\mathbb R}^n$. \\
In this section we present some sufficient conditions for the herdability of a generic matrix pair $(A,B)$. We  will later  focus on pairs $(A,B)$ 
that are endowed with specific structural properties.
\smallskip

\begin{lemma}\label{lemmaA}
Given a 
pair $(A,B), A\in {\mathbb R}^{n\times n}$ and $B\in {\mathbb R}^{n\times m}$, assume that ${\mathcal R} := {\mathcal R}(A,B)\in {\mathbb R}^{n \times nm}$ satisfies the following conditions:
\begin{itemize}
\item[i)]   ${\mathcal R}$ has no zero rows; 
\item[ii)] the set $J :=\{ j\in  [1,nm]: {\mathcal R} {\bf e}_j\  {\rm  is\   unisigned}\}$ is such that    $\lvert \cup_{j \in J} \overline{\rm ZP} ({\mathcal R} {\bf e}_j) \lvert\geq n-1$. 
\end{itemize}
Then the pair $(A,B)$ is herdable. 
\end{lemma}

\begin{proof}
Let us first suppose that $\lvert \cup_{j \in J} \overline{\rm ZP} ({\mathcal R} {\bf e}_j) \lvert = n$, which means that $\forall i \in [1,n]$, there exists $j \in J$ such that the $i$-th entry of the unisigned vector ${\mathcal R} {\bf e}_j$ is nonzero. By choosing the vector ${\bf u}$ with  entries
\begin{equation}
[{\bf u}]_j = 
\begin{cases}
0, & \text{if } j \notin J;\\ 
 {\rm sign({\mathcal R} {\bf e}_j)},  & \text{if } j \in J;
\end{cases}
\end{equation}
it is immediate to see that ${\mathcal R} {\bf u}\gg  0,$ and hence the pair $(A,B)$ is herdable.

Let us assume now that   $\lvert \cup_{j \in J} \overline{\rm ZP} ({\mathcal R} {\bf e}_j) \lvert= n -1$, and set $J = \{ j_1, j_2, \dots, j_k\}$. This implies that there exists a unique index $i \in [1,n]$ such that ${\bf e}_i^\top  {\mathcal R} [{\bf e}_{j_1}| {\bf e}_{j_2}| \dots | {\bf e}_{j_k}] = {\bf 0}_k^\top$. On the other hand, by hypothesis i), there exists $h \in [1,nm], h\notin J$, such that ${\bf e}_i ^\top {\mathcal R} {\bf e}_h \neq 0$. 
Therefore, by choosing the vector ${\bf u}$ with  entries
\begin{equation}
[{\bf u}]_j = 
\begin{cases}
 {\rm sign}({\bf e}_i^\top {\mathcal R} {\bf e}_h),  & \text{if } j = h;\\ 
0, & \text{if }   j \notin J \cup \{h\};\\
 k \cdot {\rm sign}({\mathcal R} {\bf e}_j), & \text{if }  j \in J;
\end{cases}
\end{equation}
there always exists $k \in \mathbb{R}, k> 0$, sufficiently large such that ${\mathcal R} {\bf u} \gg  0$.
$\square$
\end{proof}
\medskip

We now introduce a technical lemma, whose proof is elementary and hence omitted.
\smallskip

\begin{lemma}\label{lemmaB}
Given a matrix   $\Phi \in \mathbb{R}^{n \times k}$, assume that there exist two permutation matrices $P_1\in {\mathbb R}^{n \times n}$ and $P_2\in {\mathbb R}^{k\times k}$ such that
\be
P_1 \Phi P_2 = \begin{bmatrix}
\Phi_{11} & \Phi_{12}\cr 0 & \Phi_{22}
\end{bmatrix},
\label{triangolare}
\ee
and that both ${\rm Im}(\Phi_{11})$ and ${\rm Im}(\Phi_{22})$  include a strictly positive vector.
Then there exists ${\bf u} \in {\mathbb R}^k$ such that $\Phi {\bf u} \gg 0$. 
\end{lemma}

 Based on Lemma \ref{lemmaB}, we can derive the following sufficient condition for heardability that bears some similarities with
Lemma 3 in \cite{Ruf_arxiv}.
\\

\begin{lemma}\label{lemmaC}
Given a 
pair $(A,B), A\in {\mathbb R}^{n\times n}$ and $B\in {\mathbb R}^{n\times m}$, assume that ${\mathcal R} := {\mathcal R}(A,B)\in {\mathbb R}^{n \times nm}$  has no zero rows. 
Define the sets 
   \begin{eqnarray}
   J &:=&\{ j\in  [1,nm]: {\mathcal R} {\bf e}_j\  {\rm is\   unisigned}\}\\
   {\mathcal H} &:=& \cup_{j \in J} \overline{\rm ZP}({\mathcal R} {\bf e}_j),
   \end{eqnarray}
  and suppose that
$\forall h \in [1,n] \setminus {\mathcal H}$ 
there exists $j \in [1,nm]\setminus J$ such that
\begin{itemize}
\item[i)]  $[{\mathcal R}]_{hj} = {\bf e}_h^\top {\mathcal R} {\bf e}_j \neq 0$, and 
\item[ii)] $\forall k \in [1,n] \setminus {\mathcal H}$, condition $[{\mathcal R}]_{kj} = {\bf e}_k^\top {\mathcal R} {\bf e}_j \neq 0$ implies  
${\rm sign}([{\mathcal R}]_{kj})={\rm sign}([{\mathcal R}]_{hj})$,
\end{itemize}
  namely for every  index $h$  that does not belong to ${\mathcal H}$ there exists a column of ${\mathcal R}$, say ${\mathcal R} {\bf e}_j$,
where  the $h$-th entry  and all the nonzero entries  corresponding to indices that do not belong to ${\mathcal H}$ are  of the same sign.
Then the pair $(A,B)$ is herdable. 
\end{lemma}

\begin{proof} Under the lemma assumptions there exists a set of indices $T\subseteq [1,nm]\setminus J$ such that\\
a) $\left(\cup_{j \in J} \overline{\rm ZP}({\mathcal R} {\bf e}_j)\right) \cup \left(\cup_{j \in T} \overline{\rm ZP}({\mathcal R} {\bf e}_j)\right) = [1,n]$;\\
b)  if we denote by     $S\in {\mathbb R}^{(n-|{\mathcal H}|)\times n}$ the (selection) matrix whose rows are the $n$-dimensional canonical vectors   indexed in 
$[1,n] \setminus {\mathcal H}$, then 
$S {\mathcal R} {\bf e}_j$ is unisigned for every $j\in T$.\\
This implies that there exist two permutation matrices $P_1\in{\mathbb R}^{n\times n}$ and $P_2\in {\mathbb R}^{nm\times nm}$ such that
$$P_1{\mathcal R} P_2 = \begin{bmatrix}
{\mathcal R}_{11} & {\mathcal R}_{12}\cr 0 & {\mathcal R}_{22}
\end{bmatrix},$$
where   ${\mathcal R}_{11}$ has  all unisigned columns, while ${\mathcal R}_{22}$ has a subset of its columns that are unisigned     and therefore
${\rm Im}({\mathcal R}_{22})$  includes a strictly positive vector.
So, the result follows from Lemma \ref{lemmaB}.
$\square$
\end{proof}
\medskip

The idea behind Lemma \ref{lemmaB} and Lemma  \ref{lemmaC} can be recursively iterated, thus leading to an algorithm that checks a sufficient condition for   herdability.
The algorithm   receives as input the controllability matrix  and returns, if the sufficient condition is verified,
a confirmation that the pair $(A,B)$ is herdable. In detail, it
  proceeds as follows: 
  at each   step the algorithm detects a column vector that
  is unisigned, then sets to zero all the rows of ${\mathcal R}$ that correspond to the nonzero entries     (the non-zero pattern) of such a column vector. Subsequently, the algorithm
  repeats the same step on the modified matrix  ${\mathcal R}$, until either ${\mathcal R}$
becomes the zero matrix or the matrix ${\mathcal R}$ has no unisigned columns,   thus iteratively applying the same strategy     as in Lemma \ref{lemmaB}.
In the former case the pair $(A,B)$ is   herdable, in the second case the algorithm stops. 

 Algorithm \ref{algoritmo}, below, makes use of the following notation. Given a matrix ${\mathcal R}\in {\mathbb R}^{n \times nm}$ and a set ${\mathcal I}\subseteq [1,n]$, we denote by ${\mathcal R}_{\mathcal I}$ the matrix obtained from  ${\mathcal R}$ by (leaving unchanged all rows indexed in ${\mathcal I}$ and)
                  replacing every row indexed  in $[1,n]\setminus {\mathcal I}$ with the zero row. 

\small{
\begin{algorithm}[h!]  
\caption{Greedy algorithm to check herdability}  \label{algoritmo}
\begin{algorithmic}
\State ${\mathcal R} := [B | AB | \dots | A^{n-1}B]$  \Comment{Initialization}
\State ${\mathcal I} := [1,n]$
\State ${\mathcal J} := [1, n  m]$
   \While{${\mathcal I} \neq \emptyset$} \Comment{Recursive check}
       \For{$j \in {\mathcal J}$}
           \If{ ${\mathcal R} {\bf e}_j$ is unisigned}
                  \State  ${\mathcal J} = {\mathcal J} \setminus \{j\}$
		 \State ${\mathcal I} = {\mathcal I} \setminus \overline{\rm ZP}({\mathcal R}{\bf e}_j)$
                  \State ${\mathcal R} = {\mathcal R}_{{\mathcal I}}$  
		 \If{${\mathcal I} = \emptyset$}
                 \State (A,B) is herdable	 
                 \EndIf
		 \EndIf
     \EndFor
        \If{there are no unisigned column vectors in $\mathcal{R}$}
          \State stop 
      \EndIf
    \EndWhile
\end{algorithmic}
\end{algorithm} 

\normalsize

\section{Sufficient conditions for  herdability of pairs $(A,B)$ corresponding to a directed graph ${\mathcal G}(A)$ with $m$ leaders}\label{s3}

We now
consider the case when  the columns of the matrix $B$  are $m$ linearly independent    canonical vectors. It entails no loss of generality assuming that 
\be
B = \begin{bmatrix} I_m \cr 0\end{bmatrix} \in {\mathbb R}^{n\times m},
\label{B_leaders}
\ee
since we can always    permute  the 
  entries of the state vector so that this is the case.  Accordingly, we can block-partition the matrix $A\in {\mathbb R}^{n\times n}$ as follows:
\be
A= \begin{bmatrix} A_{11} & A_{12}\cr A_{21} & A_{22}\end{bmatrix},
\label{A_leaders}
\ee
where $A_{11}\in {\mathbb R}^{m\times m}$.
We want to investigate the herdability of the pairs $(A,B)$, where $A$ and $B$ are described as in \eqref{A_leaders} and  \eqref{B_leaders}, respectively.\\

One of the advantages of this set-up, that has already been    considered in \cite{She_herd,She_automatica,Meng_herd}, is that it allows to investigate the herdability of the pair $(A,B)$ by resorting to the
signed and weighted directed graph ${\mathcal G}(A)$    whose nodes are partitioned into leaders and followers, depending on whether the state variable associated to the node is endowed with an external and independent  control input (leader) or not (follower). Specifically we introduce  the following:
\smallskip

{\bf Assumption 1}: We assume that in the signed and weighted directed graph ${\mathcal G}(A)$
the first $m$ vertices, associated with the $m$ canonical vectors in $B$, represent the set
${\mathcal L}=[1,m]$   of leaders and 
the remaining vertices  are the set of followers, i.e., ${\mathcal F} =[m+1,n]$.
We let ${\mathcal F}_k$ be the set of followers whose   distance from the leaders is $k$,  $k \in [1,K]$, by this meaning 
$${\mathcal F}_k :=\{ j\in {\mathcal F}:    d({\mathcal L},j) =k\}.$$
We assume ${\mathcal F}_{K} \neq \emptyset$, ${\mathcal F}_k = \emptyset$, $k >K$. This means  that $K$ is the maximum distance from the set of leaders to a follower.  It entails no loss of generality assuming that  ${\mathcal F}_1 = [m+1,m+m_1], \dots, {\mathcal F}_k = [m+m_1+ \dots + m_{k-1}+1,m+m_1+ \dots+m_k]$,  so that $|{\mathcal L}| = m$ and $|{\mathcal F}_k| = m_k$ and  $m+m_1+\dots+m_{K} = n$. 
\smallskip

Under Assumption 1,  the controllability matrix in $K+1$ steps ${\mathcal R}_{K+1} = [B|AB| A^2B| \dots |A^{K}B] $ takes the following structure 
\begin{equation}
{\mathcal R}_{K+1} = 
\begin{bmatrix}
I_m & *&*&* & *\\
0 & \Phi_1 &* &* &* \\
& 0 & \Phi_2 & *& *\\
\vdots  & \vdots & \vdots & \ddots &* \\
0 & 0 & 0 & \dots & \Phi_{K}
\end{bmatrix}
\label{reach_k}
\end{equation}
and all the matrices $\Phi_k\in {\mathbb R}^{m_k\times m}$ have no zero rows.
\medskip

In the following some sufficient conditions for the   herdability of   a pair $(A,B)$  satisfying Assumption 1 are provided. 
\smallskip

\begin{proposition}  \label{pippo} 
  Consider a pair $(A,B)$, where $A\in {\mathbb R}^{n \times n}$ and $B\in {\mathbb R}^{n\times m}$ 
are described as in \eqref{A_leaders} and \eqref{B_leaders}, respectively.
Suppose that   Assumption 1 holds and hence the controllability matrix in $K+1$ steps of the pair is described as in \eqref{reach_k}
and all the matrices $\Phi_k\in {\mathbb R}^{m_k\times m}$ have no zero rows.\\
For every $k\in [1,K]$, introduce the sets
    \begin{eqnarray}
   J_k &:=&\{ j\in  [1,m]: \Phi_k {\bf e}_j\  {\rm is\   unisigned}\} \label{wq1}\\
   {\mathcal H}_k &:=& \cup_{j \in J_k} \overline{\rm ZP}(\Phi_k {\bf e}_j). \label{wq2}
   \end{eqnarray}
If $\forall k\in [1,K]$, one has $|{\mathcal H}_k|\ge m_k-1$,
then the pair $(A,B)$ is herdable.
\end{proposition}

\begin{proof} The result follows from   Lemma \ref{lemmaA} and Lemma \ref{lemmaB}.
  Indeed, by making use of Lemma \ref{lemmaA}, we can claim that ${\rm Im} (\Phi_k)$ includes a strictly positive vector, for every $k\in [1,K]$. But then, by recursively using Lemma \ref{lemmaB}, one can find a strictly positive vector in ${\rm Im} ({\mathcal R}_{K+1})$.
 $\square$\end{proof}
  \medskip

\begin{proposition} \label{prop} 
 Consider a pair $(A,B)$, where $A\in {\mathbb R}^{n \times n}$ and $B\in {\mathbb R}^{n\times m}$ 
are described as in \eqref{A_leaders} and \eqref{B_leaders}, respectively.
Suppose that   Assumption 1 holds and hence the controllability matrix in $K+1$ steps of the pair is described as in \eqref{reach_k}
and all the matrices $\Phi_k\in {\mathbb R}^{m_k\times m}$ have no zero rows.
For every $k\in [1,K]$, consider the sets
 $J_k$ and  ${\mathcal H}_k$
 as in \eqref{wq1} and \eqref{wq2}.
 \\
If $\forall h \in [1,m_k] \setminus {\mathcal H}_k$ 
there exists $i \in [1,m]\setminus J_k$ such that
\begin{itemize}
\item[i)]  $[\Phi_k]_{hi}= {\bf e}_h^\top \Phi_k {\bf e}_i \neq 0$, and 
\item[ii)] $\forall \ell \in [1,m_k] \setminus   {\mathcal H}_k$, condition $[\Phi_k]_{\ell i} = {\bf e}_\ell^\top \Phi_k {\bf e}_i \neq 0$ implies  
${\rm sign}([\Phi_k]_{\ell i})={\rm sign}([\Phi_k]_{hi})$,
\end{itemize}
 then the pair $(A,B)$ is herdable.
\end{proposition}

\begin{proof}The proof is obtained by repeatedly applying Lemma \ref{lemmaB} and Lemma \ref{lemmaC}.
$\square$\end{proof}
\smallskip

%
 
  Proposition \ref{ablocchi}, below, provides a method for the dimensionality reduction of the herdability problem for matrix pairs $(A,B)$, with $A$ as in \eqref{A_leaders} and $B$ as in \eqref{B_leaders}. It states that,   when there is a set of leaders among the $n$ nodes of the graph ${\mathcal G}(A)$, the herdability of the system depends only on the way  followers interact and     leaders exert their influence on their followers. How followers, in turn, ``evaluate/weight" the leaders has no influence on the herdability of the system. This result will be largely exploited in the rest of the paper.
%

\begin{proposition} \label{ablocchi}
Consider a pair $(A,B)$, where $A\in {\mathbb R}^{n \times n}$ and $B\in {\mathbb R}^{n\times m}$ 
are described as in \eqref{A_leaders} and \eqref{B_leaders}, respectively.
If we denote by ${\mathcal R}(A,B)$ the controllability matrix of $(A,B)$ and by 
${\mathcal R}(A_{22},A_{21})$ the controllability matrix of $(A_{22},A_{21})$, then
for every choice of ${\bf v}_1\in {\mathbb R}^m$ 
$${\bf v} =\begin{bmatrix} {\bf v}_1\cr {\bf v}_2\end{bmatrix} \in {\rm Im} ({\mathcal R}(A, B))
\ \ \Leftrightarrow \ \
{\bf v}_2 \in {\rm Im} ({\mathcal R}(A_{22},A_{21})).$$
Therefore
the pair $(A,B)$ is herdable if and only if the pair 
$(A_{22},A_{21})$ is herdable. 
 \end{proposition}

\begin{proof}
Since 
$${\mathcal R}(A,B) = \begin{bmatrix}
I_m & \Phi_{12} \cr
0 & \Phi_{22}\end{bmatrix}$$
where
$$\!\!\begin{array}{rl}
\Phi_{22} &:=
\left[\begin{matrix} A_{21} & A_{21} A_{11} + A_{22} A_{21} \end{matrix}\right.\\
&\\
&\left.\begin{matrix} A_{21}(A_{11}^2 + A_{12} A_{21}) + A_{22}  (A_{21} A_{11} + A_{22} A_{21}) & ...
\end{matrix}\right]  \\
&\\
&= \begin{bmatrix} 0 & I_{n-m} \end{bmatrix}
\begin{bmatrix} AB & A^2 B & \dots & A^{n-1}B\end{bmatrix},
\end{array}
$$
it is immediate to see that for every ${\bf v}_1\in {\mathbb R}^m$
$${\bf v} =\begin{bmatrix} {\bf v}_1\cr {\bf v}_2\end{bmatrix} \in {\rm Im} ({\mathcal R}(A, B))
\qquad \Leftrightarrow \qquad 
{\bf v}_2 \in {\rm Im} (\Phi_{22}),$$
so we are now reduced to prove that
${\rm Im} (\Phi_{22}) = {\rm Im} \left(\begin{bmatrix} 0 & I_{n-m} \end{bmatrix}
\begin{bmatrix} AB & A^2 B & \dots & A^{n-1}B\end{bmatrix}\right) = {\rm Im}\left({\mathcal R}(A_{22},A_{21})\right).$
\\
To prove this result we want to prove that for every $k\in [1, n-1]$
\be
\!\!\!\begin{bmatrix} 0\!\!  & \!\! I_{n-m} \end{bmatrix}
  A^{k}B \! = \!
\begin{bmatrix} A_{21} \!\! &\!\!  A_{22} A_{21} & \!\!\!\! \dots \!\!\!\! & A_{22}^{k-1} A_{21}
\end{bmatrix}  \begin{bmatrix} *\cr *\cr \vdots \cr I_{m}\end{bmatrix}\!,
\label{perk}
\ee
where $*$ denotes a real matrix (whose value is not relevant).\\
We proceed by induction on $k$. If $k=1$ the result is true since
$$
\begin{bmatrix} 0 & I_{n-m} \end{bmatrix}
  AB = A_{21}=
\begin{bmatrix} A_{21} 
\end{bmatrix} I_{m}.
$$
We assume now that the result is true for $k < \bar k$ and then   show that the result is true for $k=\bar k$.
Indeed, 
\begin{eqnarray*}
&&\begin{bmatrix} 0 & I_{n-m} \end{bmatrix}
  A^{\bar k}B = \begin{bmatrix} 0 & I_{n-m} \end{bmatrix} A
  A^{\bar k -1}B \\
  &=& 
  \begin{bmatrix} A_{21} & A_{22} 
\end{bmatrix} A^{\bar k -1}B\cr
&=& \begin{bmatrix} A_{21} & A_{22} 
\end{bmatrix} \begin{bmatrix} \Xi\cr 
\begin{bmatrix} 0 & I_{n-m} \end{bmatrix}
  A^{\bar k -1}B 
  \end{bmatrix}\\
  \\
  &=& A_{21} \Xi + A_{22} \begin{bmatrix} A_{21} & A_{22} A_{21} & \dots & A_{22}^{\bar k-2} A_{21}
\end{bmatrix} \!\! \begin{bmatrix} *\cr *\cr \vdots \cr I_{m}\end{bmatrix}\\
&=&  \begin{bmatrix} A_{21} & A_{22} A_{21} & \dots & A_{22}^{\bar k-1} A_{21}
\end{bmatrix} \begin{bmatrix} \Xi \cr *\cr \vdots \cr I_{m}\end{bmatrix}.
\end{eqnarray*}
From \eqref{perk}, applied for every $k\in [1,n-1]$, it follows that 
$$\begin{array}{l}
 \begin{bmatrix} 0 & I_{n-m} \end{bmatrix}
\begin{bmatrix} AB & A^2 B & \dots & A^{n-1}B\end{bmatrix} \\
\!\!= \!\! \begin{bmatrix} A_{21} & A_{22} A_{21} & \dots & A_{22}^{n-2} A_{21} \end{bmatrix} \!\!
\begin{bmatrix} I_m & * & \dots &*\cr
& I_m & \dots &*\cr 
&&\ddots & \vdots\cr
&&& I_m\end{bmatrix}
\end{array}
$$
and hence (by Cayley-Hamilton's theorem)
\begin{eqnarray*}
{\rm Im} (\begin{bmatrix} 0 & I_{n-m} \end{bmatrix}
\begin{bmatrix} AB & A^2 B & \dots & A^{n-1}B\end{bmatrix}) &=&\\ 
{\rm Im} (\begin{bmatrix} A_{21} & A_{22} A_{21} & \dots & A_{22}^{n-2} A_{21} \end{bmatrix}) 
&=&\\
 {\rm Im} ({\mathcal R}(A_{22},A_{21})).
 \end{eqnarray*}
 Consequently,
the pair $(A,B)$ is herdable if and only if the pair 
$(A_{22},A_{21})$ is herdable. $\square$\end{proof}
\medskip

 Proposition    \ref{ablocchi} allows to easily obtain   two results that are already available in the literature.
 As we will see in the next section, however, the consequences of Proposition \ref{ablocchi}    can be further exploited.
 \medskip

 \begin{corollary} [Proposition 1  \cite{Meng_herd}] \label{coro1}
Consider a pair $(A,B)$, where $A\in {\mathbb R}^{n \times n}$ and $B\in {\mathbb R}^{n\times m}$ 
are described as in \eqref{A_leaders} and \eqref{B_leaders}, respectively.
If the directed graph ${\mathcal G}(A)$ is strongly connected and structurally balanced, and the classes in which the agents split are  
${\mathcal V}_1= [1,m]$  and ${\mathcal V}_2=[m+1,n]$, then the pair $(A,B)$ is herdable. 
 \end{corollary}

 \begin{proof}  We first note that as ${\mathcal G}(A)$ is   strongly connected then ${\mathcal R}(A,B)$ cannot have zero rows, therefore (see Proposition \ref{ablocchi}) also 
 ${\mathcal R}(A_{22},A_{21})$ has no zero rows.
 If ${\mathcal V}_1= [1,m]$, then $A_{21}$ is a nonpositive matrix, while $A_{22}$ is a nonnegative matrix   (see the end of the Notation part), therefore the controllability matrix of the pair 
 $(A_{22}, A_{21})$ has all negative columns and no zero rows. This ensures that  
 $(A_{22}, A_{21})$ is herdable.
 $\square$\end{proof}
\medskip

  \begin{remark} It is easily seen that the result of Corollary \ref{coro1} would still be true if the set of leaders would include ${\mathcal V}_1$ rather than coincide with it.\end{remark}
\medskip

\begin{corollary} [Theorem 1 in \cite{She_herd}]  \label{coro2}
Consider a pair $(A,B)$, where $A\in {\mathbb R}^{n \times n}$ and $B\in {\mathbb R}^{n\times m}$ 
are described as in \eqref{A_leaders} and \eqref{B_leaders}, respectively.
  If  every follower is reached by at least one of the leaders in a single step, namely through a walk of length $1$,
and  for each leader  
 the walks of length $1$   to  its followers have the same sign, then the pair $(A,B)$ is herdable. 
\end{corollary}
\medskip

\begin{proof}
By the corollary assumptions 
the matrix $A_{21}$ is devoid of zero rows 
and all its columns are either zero vectors or unisigned vectors, therefore  ${\rm Im}(A_{21})$  includes a strictly positive vector and, since  ${\rm Im}(A_{21}) \subseteq {\rm Im}({\mathcal R}(A_{22}, A_{21}))$, also ${\rm Im}({\mathcal R}(A_{22}, A_{21}))$ does. 
%
On the other hand, 
by Proposition \ref{ablocchi}, the pair $(A,B)$ is herdable if and only if the pair $(A_{22}, A_{21})$ is herdable, and this completes the proof.
$\square$\end{proof}
\medskip

%

\section{Herdability of pairs $(A,B)$ with ${\mathcal G}(A)$ an undirected tree with a single leader}\label{s4}

Let us now consider the case when $B$ is a canonical vector 
and the matrix $A$ is a symmetric real matrix whose associated undirected graph ${\mathcal G}(A)$ is acyclic, namely ${\mathcal G}(A)$ is a tree.
This corresponds to the case of a tree with a single leader and $n-1$ followers.
This case has been investigated in \cite{She_herd}, where a sufficient condition for the herdability of the pair $(A,B)$ has been provided. 
In this section we provide a  sufficient condition for herdability that is less restrictive, and in the case of trees whose followers have 
distance at most $2$ from the leader we provide necessary and sufficient conditions.
\medskip

In order to investigate the problem we 
adopt 
the following \smallskip

\noindent {\bf Assumption~2}: The graph ${\mathcal G}(A)$ is a signed, weighted, connected and acyclic undirected graph, namely a tree.
 The leader is ${\mathcal L}=\{1\}$ (and hence $B={\bf e}_1$), while the followers   split into classes, based on their   distance from the leader. The followers at distance $1$ from the leader are ${\mathcal F}_1= [2,m_1+1]$, the followers at distance $2$ from the leader are 
${\mathcal F}_2= [m_1 +2,m_1+m_2+ 1]$, and so on till the last class ${\mathcal F}_{K}= [m_1 + \dots + m_{K-1}+2,n]$, where $K$ is the maximum distance between the leader and one of its followers. 
\medskip


\begin{proposition}\label{p1} 
Consider a pair $(A,B)$, with $A\in \mathbb{R}^{n\times n}$ and $B\in {\mathbb R}^n$ satisfying the previous Assumption 2.\\
If, for every $k\in [0,K-1]$, all the edges from the vertices in ${\mathcal F}_k$ to the vertices in ${\mathcal F}_{k+1}$ have the same sign, 
then the pair $(A,B)$ is herdable.
\end{proposition}

\begin{proof}
Under the previous assumption, it is easy to see that
every vertex in ${\mathcal F}_k$ is reached for the first time  by the leader in $k$ steps,  $k \in [0,K]$, and subsequently it is reached 
after $k+2h$ steps for every $h\in \{1,2,3,\dots\}$ (since each undirected edge of the graph 
can be crossed back and forth, and hence  condition $[A^kB]_{ij}\ne 0$ implies $[A^{k+2} B]_{ij}\ne 0$). Therefore the controllability matrix of the pair $(A,B)$ takes the form
\begin{equation}\label{herd_tree}
{\mathcal R} 
= \begin{bmatrix}\
1 & 0                 & *             & 0 & * & \dots \cr
0 & {\bf v}_1      & 0            & * & 0 & \dots \cr
0 & 0                 & {\bf v}_2 & 0 & * & \dots \cr
0 & 0                 & 0 & {\bf v}_3 & 0 &   \dots \cr
\vdots &    \vdots              & \vdots &   &   \vdots&   \vdots \cr
0 & 0                 & \dots  &\dots   & {\bf v}_{K}    &  \dots 
\end{bmatrix},
\end{equation}
where ${\bf v}_k\in {\mathbb R}^{m_k}, k\in [1,K],$ are, by assumption, unisigned, while $*$ denotes (nonzero) vectors/entries whose values are not relevant.
So, by making use of Proposition \ref{pippo}, we immediately deduce that there exists a strictly positive vector in the image of ${\mathcal R}$, and hence $(A,B)$ is herdable.
$\square$\end{proof}
\medskip

\begin{remark}
Theorem 3 in \cite{She_herd}  follows   as a corollary of the previous proposition, since it imposes  that all paths from the leader to the followers in 
${\mathcal V}_o := \cup_{h \in {\mathbb Z}_+} {\mathcal F}_{1+2h}$ have the same sign and, at the same time,  
all paths from the leader to the followers in 
${\mathcal V}_e := \cup_{h \in {\mathbb Z}_+} {\mathcal F}_{2+2h}$ have the same sign. This means that not only all the  edges from vertices in ${\mathcal F}_k$ to 
vertices in ${\mathcal F}_{k+1}, k\in [0,K-1]$,  (where ${\mathcal F}_0:={\mathcal L}$) have the same signs, but such signs are uniquely determined for $k\ge 1$ once we choose the signs of the edges from ${\mathcal F}_1$ to  
${\mathcal F}_2$.
\end{remark}
\smallskip

\begin{example} \label{ex1}
Consider a pair $(A,B)$, with  $A=A^\top\in {\mathbb R}^{9\times 9}$ and $B = {\bf e}_1$, and assume that  the   undirected graph ${\mathcal G}(A)$ associated with the matrix $A$ is a tree    whose structure and   edge signs are described in Figure \ref{grafoEx1}.  
\begin{figure}[h!]  
\includegraphics[scale=0.7]{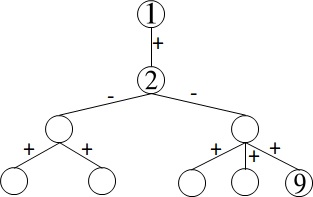}
\centering
\caption{  Tree structure of the  herdable system of Example \ref{ex1}.} \label{grafoEx1}
\end{figure}
 The nodes $i=2$ and $j=9$ both belong to $\mathcal{V}_o$, since both of them are reached from the leader (node $1$ in Fig. \ref{grafoEx1}) in an odd number of steps ($1+2h$ and $3+2h$, $h \in \{0,1,2, \dots\}$, respectively). The node $i$ is reached by the leader with positive walks, while $j$ with negative ones, so the hypotheses of Theorem 3 in \cite{She_herd} are violated. However, the controllability matrix of the pair takes the structure in \eqref{herd_tree} for $K=3$, with unisigned vectors 
 ${\bf v}_1$, ${\bf v}_2$ and ${\bf v}_3$, the first one with a positive entry, while the other two with negative entries,
 thus the pair is herdable by Proposition \ref{p1}.
\end{example}

  Given a matrix $A$ and hence a graph ${\mathcal G}(A)$ with a tree structure,
  we propose now   Algorithm \ref{algoritmo2} for the selection of a (unique) leader $i$   in order to  ensure, if possible,   that the pair $(A, {\bf e}_i)$ is herdable.
 The algorithm     searches for a single node, if it exists, for which the sufficient condition given in  Proposition 
\ref{p1} is satisfied.
For the meaning of the symbols ${\rm Out}_{+}({\mathcal F}), {\rm Out}_{-}({\mathcal F})$ etc., we refer the reader to the Notation part, at the beginning of the paper. 

%

\small{
\begin{algorithm}[h!]  
\caption{Algorithm for the selection of a single leader   to ensure  herdability of a pair $(A,B)$ when ${\mathcal G}(A)$ is a tree} \label{algoritmo2}
\begin{algorithmic}
   \For{${i} \in \mathcal{V}$} 
       \If{{\small ${\rm Out}_{+}(i) = {\rm Out}(i)  \ne \emptyset$ or ${\rm Out}_{-}(i) = {\rm Out}(i)  \ne \emptyset$}}
          \State ${\mathcal L}  :=  \{i\}$
           \State ${\mathcal F} :=  \{ j : (i,j) \in \mathcal{E} \}$
          \State ${\mathcal H} :=  {\mathcal L} \cup {\mathcal F}$ 
           \If{$|{\mathcal H}| = n$}
                \State $(A,B)$ is herdable
            \Else
             \While{{\small ${\rm Out}_{+}({\mathcal F})={\rm Out}({\mathcal F}) \ne \emptyset$ or \\
 \qquad \qquad  \qquad \qquad ${\rm Out}_{-}({\mathcal F}) =  {\rm Out}({\mathcal F})  \ne  \emptyset$}}
                     \State ${\mathcal F} = {\mathcal F}\cup{\rm Out}({\mathcal F})$
                     \State ${\mathcal H} = {\mathcal H} \cup {\mathcal F}$
                      \If{$|{\mathcal H}| = n$}
                          \State $(A,B)$ is herdable
			\EndIf
                   \EndWhile

		\EndIf
		\EndIf
           \EndFor
\end{algorithmic}
\end{algorithm}   
}
\normalsize
\smallskip

 Propositions \ref{p2} and \ref{p3}, below, provide complete characterizations of herdability for trees in which followers have all distance $1$ from the leader or distance at most $2$, respectively.

\begin{proposition}\label{p2} 
Consider a pair $(A,B)$, with $A\in \mathbb{R}^{n\times n}$ and $B\in {\mathbb R}^n$ satisfying   Assumption 2, and suppose that all the followers have distance one from the leader.\\
Then the pair $(A,B)$ is herdable if and only if all the edges have the same sign.
\end{proposition}

\begin{proof}
If all the followers have distance $1$ from the leader, namely $K=1$, then
$$A = \begin{bmatrix} 0 & A_{12}\cr A_{21} & {\bf 0}_{(n-1)\times (n-1)}\end{bmatrix},$$
where $A_{12} = A_{21}^\top \in {\mathbb R}^{1\times (n-1)}$ is  devoid of zero entries.
By   Proposition \ref{ablocchi}, $(A,B)$ is herdable if and only if the pair $({\bf 0}_{(n-1)\times (n-1)}, A_{21})$ is herdable, and this is the case if and only if 
$A_{21}$ is unisigned.
$\square$\end{proof}
\smallskip

\begin{proposition}\label{p3} 
Consider a pair $(A,B)$, with $A\in \mathbb{R}^{n\times n}$ and $B\in {\mathbb R}^n$ satisfying Assumption~2, and suppose that all the followers have distance at most $2$ from the leader, and hence
$$A = \begin{bmatrix} 0 & A_{12} & {\bf 0}_{1\times m_2}\cr A_{21} & {\bf 0}_{m_1\times m_1} & A_{23}\cr
{\bf 0}_{m_2\times 1} & A_{32} & {\bf 0}_{m_2 \times m_2}\end{bmatrix},$$
where $A_{21} = A_{12}^\top\in {\mathbb R}^{m_1}$ and $A_{32}= A_{23}^\top \in {\mathbb R}^{m_2\times m_1}$. 
Then the pair $(A,B)$ is herdable if and only if 
for every $i,j\in {\mathcal F}_1= [2, m_1+1]$ (including $i=j$)\footnote{Note that for $i=j$  condition i) becomes trivial, while condition ii) becomes ``$A_{32} {\bf e}_i$ is unisigned".} such that 
\be
 [A_{23}A_{32}]_{ii}= \sum_{k=1}^{m_2} \left([A_{32}]_{ki}\right)^2 = \sum_{k=1}^{m_2} \left([A_{32}]_{kj}\right)^2 =  [A_{23}A_{32}]_{jj},
\label{noncompatta}
\ee
(namely for every pair $(i,j)\in {\mathcal F}_1\times {\mathcal F}_1$ such that the sum of the squares of the weights of all edges $(i,k), k\in{\mathcal F}_2$, coincides 
with the sum of the squares of the weights of all edges $(j,k), k\in{\mathcal F}_2$) we have:
\begin{itemize}
\item[i)] $[A_{21}]_i \cdot [A_{21}]_j >0$ (namely 
 the two edges from the leader ${\mathcal L}$ to  $i$ and $j$ have the same sign);
 \item[ii)] $A_{32} ({\bf e}_i + {\bf e}_j)$ is either zero or unisigned (namely all edges from $i$ and $j$ to their followers in ${\mathcal F}_2$ have the same sign).
 \end{itemize}
\end{proposition}

\begin{proof}
 First of all, we highlight that, by Assumption 2, $\Gamma := A_{21}$ is devoid of zero entries, 
and for every $i\in [1,m_2]$ the row vector ${\bf e}_i^\top A_{32}$ is a monomial vector (namely it has a single nonzero entry). Consequently, $\Lambda:= A_{23}A_{32}$ is a diagonal matrix (with nonnegative diagonal entries).

\noindent By   Proposition \ref{ablocchi}, $(A,B)$ is herdable if and only if the pair 
$$\left(\begin{bmatrix} {\bf 0}_{m_1\times m_1} & A_{23}\cr
A_{32} & {\bf 0}_{m_2\times m_2}\end{bmatrix}, \begin{bmatrix}A_{21}\cr {\bf 0}_{m_2}\end{bmatrix}\right)$$ 
is herdable, and this is the case if and only if the image of 
the controllability matrix  $\hat {\mathcal R}$ of the previous pair, given in \eqref{reachhat}
\begin{figure*}
\be
\hat {\mathcal R} :=
\begin{bmatrix}
A_{21} & 0 & (A_{23} A_{32}) A_{21} & 0 & (A_{23} A_{32})^2 A_{21} & 0 & \dots \cr
0 & A_{32} A_{21} & 0 & A_{32} (A_{23} A_{32}) A_{21} & 0 & A_{32} (A_{23} A_{32})^2 A_{21}  & \dots
\end{bmatrix}
\label{reachhat}
\ee
 \begin{center}
-------------------------------------------------------------------------------------------------------------------------------------------------
\end{center}
\end{figure*}
\noindent includes a strictly positive vector.
This is the case if and only if the following two conditions simultaneously hold:
\begin{itemize}
\item[a)] the image of the controllability matrix
${\mathcal R}_1 :=
\begin{bmatrix}
A_{21} &   (A_{23} A_{32}) A_{21}   & (A_{23} A_{32})^2 A_{21}   & \dots \end{bmatrix}$ includes a strictly positive vector, namely the pair 
$(A_{23} A_{32}, A_{21})$ is herdable;
\item[b)] 
the image of the matrix
 $A_{32}  {\mathcal R}_1$
includes a strictly positive vector.
\end{itemize}
As the matrix $\Lambda = A_{23} A_{32}$ 
 is  diagonal, while the column vector $\Gamma = A_{21}$ has no zero entries,
 by Lemma
\ref{vandermonde}, the pair $(\Lambda,\Gamma)=(A_{23} A_{32}, A_{21})$ is herdable if and only if 
condition
 \eqref{noncompatta}
 implies $[A_{21}]_i \cdot [A_{21}]_j >0$. This means that a) is equivalent to condition i).
\\
 So, we are now remained with proving that if i) (equivalently, a)) holds, then b) and ii) are equivalent.
If i) holds,  by referring to the proof of Lemma \ref{vandermonde}, we can assume without loss of generality that $\Lambda$ and $\Gamma$ take the form given in \eqref{lambdagamma} and 
 claim that 
 $${\rm Im} \left(A_{32}  {\mathcal R}_1\right) =
 {\rm Im} \left(A_{32} \cdot {\rm diag}
\{{\bm \gamma}_1, \dots,  {\bm \gamma}_p, {\bm \gamma}_{p+1}, \dots, {\bm \gamma}_s\}\right),$$
where ${\bm \gamma}_i\in {\mathbb R}^{n_i}$ is strictly positive if $i\in [1,p]$ and strictly negative if $i\in [p+1,s]$.\\
Set
$W = 
\begin{bmatrix}
 {\bf w}_1 \,| \, \dots \,|\, {\bf w}_p  \,|\, {\bf w}_{p+1} \,|\, \dots  \,|\, {\bf w}_s
\end{bmatrix} := A_{32} \cdot {\rm diag}
\{{\bm \gamma}_1, \dots,  {\bm \gamma}_p, {\bm \gamma}_{p+1}, \dots, {\bm \gamma}_s\},
$
where
each vector ${\bf w}_i$ is obtained 
by combining with the coefficients  of the vector ${\bm \gamma}_i$ (having all the same sign) the columns of $A_{32}$ of indices
$[h_i+1, h_i+n_i]$, where by definition $h_1:=0$, while
 $h_i := n_1 +n_2 +\dots + n_{i-1}$ for $i\in [2,s]$.\\
 We observe that  
 all  columns of $A_{32}$ are either zero (if a vertex in ${\mathcal F}_1= [2, m_1+1]$ has no followers) or have disjoint nonzero patterns, meaning that for every $\ell, m\in [h_i+1, h_i+n_i], \ell\ne m,$
 $\overline{\rm ZP}(A_{32} {\bf e}_\ell) \cap \overline{\rm ZP}(A_{32} {\bf e}_m) =\emptyset$.
 As a result also the columns ${\bf w}_i$ of $W$ are either zero or have disjoint nonzero patterns. \\
We can now conclude that condition b) holds if and only if   ${\rm Im} \left(A_{32}  {\mathcal R}_1\right) =
 {\rm Im} (W)$  contains a strictly positive 
 vector, but this is possible if and only if all vectors ${\bf w}_h$ are unisigned.
 By the way the vectors ${\bf w}_h$ have been obtained, this is possible if and only if condition ii) holds.
$\square$\end{proof}
\medskip

\section{conclusions} \label{s5}
In this paper herdability  of linear time-invariant systems has been investigated.  Special attention has been given to pairs $(A,B)$ 
corresponding to  leader-follower networks ${\mathcal G}(A)$, and to networks with tree topologies and a single leader.
For this latter case, an algorithm for leader selection is provided, and  when the  distance from the leader to  the followers is  at most $2$, necessary and sufficient conditions for herdability are stated. Future research will focus on the study of herdability for networked systems interacting through more general topological structures.
\medskip

\section*{Technical lemma}

 \begin{lemma} \label{vandermonde}
 Given a matrix pair $(\Lambda, \Gamma)$, with $\Lambda = {\rm diag}\{\lambda_1, \lambda_2, \dots, \lambda_n\} \in {\mathbb R}^{n\times n}$ a diagonal matrix,
  and $\Gamma\in {\mathbb R}^{n}$ devoid of zero entries, the pair is herdable if and only 
  $\lambda_i=\lambda_j$ 
  implies $[\Gamma]_i\cdot [\Gamma]_j > 0$, namely the $i$-th and the $j$-th entries of $\Gamma$ have the same sign.
  \end{lemma}
  
  \begin{proof} We first prove that if the pair $(\Lambda, \Gamma)$ is herdable, then $\lambda_i=\lambda_j$ 
  implies $[\Gamma]_i\cdot [\Gamma]_j > 0$.
  Suppose, by contradiction, that $\lambda_i=\lambda_j=: \lambda$ and  $[\Gamma]_i\cdot [\Gamma]_j < 0$.
  Then it is easy to see that $i\ne j$ and the vector
  ${\bf w}^\top := [\Gamma]_j {\bf e}_i^\top - [\Gamma]_i {\bf e}_j^\top$ 
  satisfies 
  ${\bf w}^\top \Lambda = \lambda {\bf w}^\top,$
  namely ${\bf w}$ is a (left) eigenvector of $\Lambda$ corresponding to $\lambda$,
  and ${\bf w}^\top \Gamma=0$. Consequently, it is immediate to prove that  ${\bf w}$ is orthogonal to ${\rm Im}(\mathcal{R}(\Lambda, \Gamma))$, i.e. 
  ${\bf w}^\top {\mathcal R}(\Lambda, \Gamma)={\bf 0}_n^\top.$
  Since ${\bf w}^T$  is unisigned (since $[\Gamma]_j$ and $- [\Gamma]_i$ have the same sign), it is impossible that there exists a strictly positive vector
  ${\bf v}\in {\rm Im} ({\mathcal R}(\Lambda, \Gamma))$, since this would imply ${\bf w}^\top {\bf v} \ne 0$. Therefore
  the pair $(\Lambda, \Gamma)$ cannot be herdable.
  \smallskip
  
  We now prove that if $\lambda_i=\lambda_j$ 
  implies $[\Gamma]_i\cdot [\Gamma]_j > 0$, then the pair $(\Lambda, \Gamma)$ is herdable.
  \\
 If all entries of $\Gamma$ have the same sign, the pair  $(\Lambda, \Gamma)$ is trivially herdable. So, suppose this is not the case.
  It entails no loss of generality to first permute the entries of $\Gamma$ (and accordingly the rows and columns of $\Lambda$)
  so that the first ones are positive and the last ones are negative. Then we can permute the entries in such a way that
 the identical diagonal entries of $\Lambda$ are consecutive.
 This implies that, under the previous assumptions,  $\Lambda$ and $\Gamma$ take the following form
\be
\!\!{\small \Lambda =
 \begin{bmatrix}
 \Lambda_1 & & &\vline& & & \cr
 &\ddots & &\vline&&&\cr
 & & \Lambda_p&\vline &&&\cr
 \hline
 & & & \vline &\Lambda_{p+1} & & \cr
  & & & \vline& &\ddots& \cr
 & & & \vline& && \Lambda_s 
 \end{bmatrix}}
 \
 {\small\Gamma = \begin{bmatrix}
 {\bm \gamma}_1\cr \vdots\cr {\bm \gamma}_p\cr \hline
 {\bm \gamma}_{p+1}\cr \vdots\cr {\bm \gamma}_s
\end{bmatrix}\!\!}
\label{lambdagamma}
\ee

\normalsize

\noindent   where each $\Lambda_i$ is a scalar matrix of size say $n_i$, namely $\Lambda_i=\bar \lambda_i I_{n_i}$ with $\bar \lambda_i\in \{\lambda_1, \dots, \lambda_n\}$,
while ${\bm \gamma}_i\in {\mathbb R}^{n_i}$ is a strictly positive vector for every $i\in [1,p]$ and a strictly negative vector for every $i\in [p+1,s]$. 
Moreover, by the assumption that $\lambda_i=\lambda_j$ 
  implies $[\Gamma]_i\cdot [\Gamma]_j > 0$ 
   we can claim that $\bar \lambda_h\ne \bar \lambda_k$ for $h\ne k$.
It is immediate to see that the controllability matrix of the pair $(\Lambda, \Gamma)$ factorizes as in \eqref{factor}.
\be
\small{ {\mathcal R}(\Lambda, \Gamma)\!=\! 
\begin{bmatrix}
{\bm \gamma}_1 & & &\!\!\!\!\!\!\vline\!\!\!\!\!\!& & & \cr
 &\!\!\!\!\!\! \ddots \!\!\!\!\!\!  & &\!\!\!\!\!\!\vline\!\!\!\!\!\!&&&\cr
 & & {\bm \gamma}_p&\!\!\!\!\!\!\vline\!\!\!\!\!\!&&&\cr
 \hline
 & & &\!\!\!\!\!\!\vline\!\!\!\!\!\!& {\bm \gamma}_{p+1} & & \cr
  & & &\!\!\!\!\!\!\vline\!\!\!\!\!\!&\!\!\!\!\!\!  &\!\!\!\!\!\!  \ddots\!\!\!\!\!\! & \cr
 & & &\!\!\!\!\!\!\vline\!\!\!\!\!\!& &&  {\bm \gamma}_s 
 \end{bmatrix}
 \begin{bmatrix}
 1\!\!\!\!\!\!  & \bar \lambda_1\!\!\!\!\!\! & \dots &\!\!\!\!\!\! \bar\lambda_1^{n-1}\cr
\vdots\!\!\!\!\!\! &\vdots\!\!\!\!\!\! &\dots &\!\!\!\!\!\!\vdots\cr
 1 \!\!\!\!\!\! & \bar\lambda_p\!\!\!\!\!\! & \dots & \!\!\!\!\!\!\bar\lambda_p^{n-1}\cr
  1\!\!\!\!\!\!  & \bar\lambda_{p+1}\!\!\!\!\!\! & \dots &\!\!\!\!\!\! \bar\lambda_{p+1}^{n-1}\cr
\vdots\!\!\!\!\!\! &\vdots\!\!\!\!\!\! &\dots &\!\!\!\!\!\!\vdots\cr
 1\!\!\!\!\!\!  & \bar\lambda_s\!\!\!\!\!\! & \dots &\!\!\!\!\!\! \bar\lambda_s^{n-1}
 \end{bmatrix}\!\!
\label{factor}}
\ee


\noindent Since $\bar\lambda_1, \dots, \bar\lambda_s$ are all distinct, the Vandermonde matrix on the right of \eqref{factor} is of full row rank. This ensures that
$${\rm Im}  ({\mathcal R}(\Lambda, \Gamma)) =
{\rm Im}\left(\begin{bmatrix}
{\bm \gamma}_1 & & &\vline& & & \cr
 &\ddots & &\vline&&&\cr
 & & {\bm \gamma}_p&\vline &&&\cr
 \hline
 & & & \vline & {\bm \gamma}_{p+1} & & \cr
  & & & \vline& &\ddots& \cr
 & & & \vline& && {\bm \gamma}_s 
 \end{bmatrix}\right)$$
 and since all columns of this latter matrix are unisigned, it is immediate to see
 that there exists a strictly positive vector in its image, and hence
the pair $(\Lambda, \Gamma)$ is herdable.  $\square$
\end{proof}

\bibliographystyle{plain} 

 \bibliography{Refer162}

\begin{thebibliography}{10}

\bibitem{antsaklis2007special}
P.~Antsaklis and J.~Baillieul.
\newblock Special issue on technology of networked control systems.
\newblock {\em Proc. IEEE}, 95 (1):5--8, 2007.

\bibitem{comp_gen}
James~M. Bower and Hamid Bolouri, editors.
\newblock {\em Computational Modeling of Genetic and Biochemical Networks}.
\newblock MIT Press, 2001.

\bibitem{BookFarina}
L.~Farina and S.~Rinaldi.
\newblock {\em Positive linear systems: theory and applications}.
\newblock Wiley-Interscience, Series on Pure and Applied Mathematics, New York,
  2000.

\bibitem{gupta}
S.~Gupta, S.~S. Bisht, R.~Kukreti, S.~Jain, and S.~K. Brahmachari.
\newblock Boolean network analysis of a neurotransmitter signaling pathway.
\newblock {\em J. Theoret. Biol.}, 244:463--469, 2007.

\bibitem{Jacquez}
J.A. Jacquez.
\newblock {\em Compartmental analysis in biology and medicine}.
\newblock Elsevier, Amsterdam (NL), 1972.

\bibitem{LiZhuDing}
W.~Li, Z.~Zhu, and S.X. Ding.
\newblock Fault detection design of networked control systems.
\newblock {\em IET Control Theory \& Applications}, 5 (12):1439 -- 1449, 2011.

\bibitem{Meng_herd}
S.~Meng, B.~She, H.~Gao, and Z.~Kan.
\newblock Leader group selection for herdability of structurally balanced
  signed networks.
\newblock {\em Conference on Decision and Control (CDC) 2020, Jeju island,
  Republic of Korea}, pages 5567--5572, 2020.

\bibitem{struct_controllability}
S.~S. Mousavi, M.~Haeri, and M.~Mesbahi.
\newblock Strong structural controllability of signed networks.
\newblock {\em Conference on Decision and Control (CDC) 2019, Nice, France},
  pages 4557--4562, 2019.

\bibitem{ParlangeliNotarstefano}
G.~Parlangeli and G.~Notarstefano.
\newblock On the reachability and observability of path and cycle graphs.
\newblock {\em IEEE Trans. Automatic Control}, 57 (3):743 -- 748, 2012.

\bibitem{graph_controllability}
A.~Rahmani, J.~Meng, M.~Mehran, and M.~Egerstedt.
\newblock Controllability of multi-agent systems from a graph-theoretic
  perspective.
\newblock {\em SIAM J. Control Optim.}, 48(1):162--186, 2009.

\bibitem{Ruf_Shamma}
S.F Ruf, M~Egerstedt, and J.S. Shamma.
\newblock Herdable systems over signed, directed graphs.
\newblock {\em Annual American Control Conference (ACC) 2018, Milwaukee, USA},
  pages 1807--1812, 2018.

\bibitem{Ruf_arxiv}
S.F Ruf, M~Egerstedt, and J.S. Shamma.
\newblock Herdability of linear systems bases on sign patterns and graph
  structures.
\newblock {\em arXiv:1904.08778}, 2019.

\bibitem{sociology1}
J.~Scott.
\newblock Social network analysis.
\newblock {\em Sociology}, 22(1):397--411, 1988.

\bibitem{She_herd}
B.~She, M.~Cai, and Z.~Kan.
\newblock Characterizing herdability of signed networks via graph walks.
\newblock {\em Conference on Decision and Control (CDC) 2019, Nice, France},
  pages 5456--5461, 2019.

\bibitem{She_automatica}
B.~She and Z.~Kan.
\newblock Characterizing controllable subspace and herdability of signed
  weighted networks via graph partition.
\newblock {\em Automatica}, 115:1--7, 2020.

\bibitem{sign_controllability}
M.~Tsatsomeros.
\newblock Sign controllability: Sign patterns that require complete
  controllability.
\newblock {\em SIAM J. Matrix Anal. Appl.}, 19(2):355--364, 1998.

\bibitem{TACconPaolo}
M.E. Valcher and P.~Santesso.
\newblock Reachability properties of single-input continuous-time positive
  switched systems.
\newblock {\em IEEE Trans. Automatic Control}, 55, no.5:1117--1130, 2010.

\end{thebibliography}

\end{document}